\documentclass[10pt]{amsart}

\usepackage{amsmath}
  \usepackage{paralist}
  \usepackage{graphics} 
  \usepackage{epsfig} 
\usepackage{graphicx}  \usepackage{epstopdf}

 \usepackage[colorlinks=true]{hyperref}
\hypersetup{urlcolor=blue, citecolor=red}

\usepackage{tikz}
\usetikzlibrary{arrows, automata}
\usetikzlibrary{calc}
\usetikzlibrary{intersections}

\tikzset{
    right angle quadrant/.code={
        \pgfmathsetmacro\quadranta{{1,1,-1,-1}[#1-1]}     
        \pgfmathsetmacro\quadrantb{{1,-1,-1,1}[#1-1]}},
    right angle quadrant=1, 
    right angle length/.code={\def\rightanglelength{#1}},   
    right angle length=1ex, 
    right angle symbol/.style n args={3}{
        insert path={
            let \p0 = ($(#1)!(#3)!(#2)$) in     
                let \p1 = ($(\p0)!\quadranta*\rightanglelength!(#3)$), 
                \p2 = ($(\p0)!\quadrantb*\rightanglelength!(#2)$) in 
                let \p3 = ($(\p1)+(\p2)-(\p0)$) in  
            (\p1) -- (\p3) -- (\p2)
        }
    }
}

\newtheorem{theorem}{Theorem}[section]

\newtheorem{lemma}[theorem]{Lemma}
\newtheorem{proposition}[theorem]{Proposition}
\newtheorem{conjecture}[theorem]{Conjecture}

\newtheorem{question}[theorem]{Question}
\theoremstyle{definition}
\newtheorem{definition}[theorem]{Definition}
\newtheorem{remark}[theorem]{Remark}

\title[Running heading with forty characters or less]
      {On the ergodicity of geodesic flows on surfaces of nonpositive curvature}

\author[first-name1 last-name1 and first-name2 last-name2]{Weisheng Wu}


\address{Department of Mathematics, Pennsylvania State University, University Park, PA 16802, USA}
 \email{wu@math.psu.edu}

\begin{document}
\maketitle

\markboth{On the ergodicity of geodesic flows on surfaces of nonpositive curvature}
{On the ergodicity of geodesic flows on surfaces of nonpositive curvature}
\renewcommand{\sectionmark}[1]{}

\begin{abstract}
Let $M$ be a smooth compact surface of nonpositive curvature, with genus $\geq 2$. We prove the ergodicity of the geodesic flow on the unit tangent bundle of $M$ with respect to the Liouville measure under the condition that the set of points with negative curvature on $M$ has finitely many connected components. Under the same condition, we prove that a non closed "flat" geodesic doesn't exist, and moreover, there are at most finitely many flat strips, and at most finitely many isolated closed "flat" geodesics.
\end{abstract}

\maketitle
\section{Introduction}

Let $M$ be a smooth, connected, compact surface without boundary, with genus $g \geq 2$, and of nonpositive curvature. The geodesic flow $\Phi^t$, is defined on the unit tangent bundle $T^1M$. It is well known that when the curvature of the surface is strictly negative, the geodesic flow is Anosov, and its ergodicity with respect to the Liouville measure $\nu$ can be proved by the Hopf argument (cf., for example \cite{BS}). However, for surfaces of nonpositive curvature, the ergodicity of the geodesic flow is not known yet. The dynamical behavior of the flow gets more complicated because of the existence of the "flat geodesics" defined as follows. We define:
$$\Lambda:=\{x\in T^1M: K(\gamma_x(t))\equiv 0, \ \forall t\in \mathbb{R}\}$$
where $K$ denotes the curvature of the point, and $\gamma_x(t)$ denotes the unique geodesic on $M$ with an initial velocity $x\in T^1M$. we call $\gamma_x$ a "flat" geodesic if $x\in \Lambda$, i.e., the curvature along the geodesic is always zero. It is proved that the geodesic flow is Anosov if and only if $\Lambda=\emptyset$ (cf. \cite{E}), and in this case the ergodicity follows from the Hopf argument.

By Pesin's well-known result (cf. \cite{BP}), the geodesic flow is ergodic on the following set:
\begin{equation}\label{e:Pesinset}
\Delta:=\{x\in T^1M: \limsup _{t\to\infty} \frac{1}{t}\int_0^t K(\gamma_x(s))ds <0\}.
\end{equation}
Clearly $\Delta \subset \Lambda^c$. It is stated in \cite{BM} that the geodesic flow is also ergodic on $\Lambda^c$. Indeed, we have
\begin{lemma}
$\nu(\Lambda^c \setminus \Delta)=0$.
\end{lemma}
\begin{proof}
Assume $\nu(\Lambda^c \setminus \Delta)>0$. Let $\pi: T^1M \to M$ be the natural projection. Denote $f(x):=\chi_{\Lambda^c \setminus \Delta}(x)\cdot K(\pi(x))$. Note that $f(x) \leq 0$. By Birkhoff Ergodic Theorem, for $\nu$-a.e. $x\in T^1M$,
$$\lim_{t \to \infty}\frac{1}{t}\int_0^t f(\Phi^s(x))ds :=\tilde{f}(x)$$
and
\begin{equation}\label{e:integral}
\int_{T^1M}\tilde{f}(x)d\nu(x)=\int_{T^1M}f(x)d\nu(x) \leq 0.
\end{equation}
By the definition of $\Delta$ in \eqref{e:Pesinset}, $\tilde{f}(x)=0$ for $\nu$-a.e. $x \in T^1M$. Then by \eqref{e:integral}, $\int_{T^1M}f(x)d\nu(x)=0$, so $f(x)=0$ for $\nu$-a.e. $x \in T^1M$. Hence, $K(\pi(x))=0$ for $\nu$-a.e. $x \in \Lambda^c \setminus \Delta$. Since the orbit foliation of $\Phi^t$ is smooth, for $\nu$-a.e. $x \in \Lambda^c \setminus \Delta$, one has $K(\Phi^t(x))=0$ for a.e. $t$. By continuity of the curvature function $K$, we have $K(\Phi^t(x))\equiv0$ for $\forall t \in \mathbb{R}$, i.e., $x\in \Lambda$, a contradiction to $x \in \Lambda^c \setminus \Delta$. Therefore, $\nu(\Lambda^c \setminus \Delta)=0$.
\end{proof}

 So the geodesic flow is ergodic on the set $\Lambda^c$. Therefore, the geodesic flow is ergodic on $T^1M$ if $\nu(\Lambda)=0$. It is not known in general if $\nu(\Lambda)=0$, but this is the case for all the known examples so far. Moreover, in all these examples, the flat geodesics are always closed. This motivates the following conjecture whose statement is stronger than ergodicity (cf. \cite{RH}):

\begin{conjecture}\label{conjecture}
All flat geodesics are closed and there are only finitely many homotopy classes of such geodesics. In particular, $\nu(\Lambda)=0$ and hence the geodesic flow is ergodic.
\end{conjecture}

In this paper we prove the following two theorems according to the dichotomy: (1) $\Lambda \subset \text{Per}(\Phi)$; \ (2)$\Lambda \cap (\text{Per\ } (\Phi))^c \neq \emptyset$. Here $\text{Per}(\Phi)$ denotes the set of periodic points of the geodesic flow, and $\mathcal{O}(z)$ will denote the orbit of $z$ under the geodesic flow.

\begin{theorem}\label{A}
If $\Lambda \subset \text{Per}(\Phi)$, then
$$\Lambda = \mathcal{O}_1 \cup \mathcal{O}_2 \cup \ldots \mathcal{O}_k \cup \mathcal{F}_1\cup \mathcal{F}_2 \cup \ldots \cup \mathcal{F}_l,$$
where each $\mathcal{O}_i, 1\leq i \leq k$ is an isolated periodic orbit and each $\mathcal{F}_j, 1\leq j \leq l$ consists of vectors tangent to a flat strip. Here $k$ or $l$ are allowed to be $0$ if there is no isolated closed flat geodesic or no flat strip.
\end{theorem}

\begin{theorem}\label{B}
If $\Lambda \cap (\text{Per\ } (\Phi))^c \neq \emptyset$, then there exist $y, z \in \Lambda$, $y\notin \mathcal{O} (z)$, such that
$$d(\Phi^t(y), \Phi^t(z))\to 0,  \ \ \text{as\ } t\to +\infty.$$
\end{theorem}

In the process of proving the above two theorems, we obtain a result of independent importance:
\begin{theorem}\label{C}
$\Lambda \cap (\text{Per\ } (\Phi))^c$ is a closed set in $\Lambda$.
\end{theorem}
Theorem \ref{C} says that if we count a flat strip as a single orbit then closed flat orbits must be isolated from non-closed flat orbits.

Now let $\{p\in M: K(p)<0\}$ be the set of points with negative curvature on $M$. As a consequence of Theorem \ref{A} and \ref{B}, we can prove the Conjecture \ref{conjecture} in the case when $\{p\in M: K(p)<0\}$ has only finitely many connected components:

\begin{theorem}\label{Main}
If the set $\{p\in M: K(p)<0\}$ has finitely many connected components, then $\Lambda \subset \text{Per}(\Phi)$. In particular, the geodesic flow is ergodic.
\end{theorem}

Theorem \ref{Main} gives a negative answer to Question 6.2.1 asked by Burns in a recent survey \cite{BM}, for the case when $\{p\in M: K(p)<0\}$ has only finitely many connected components. Furthermore, by Theorem \ref{A} there are at most finitely many flat strips and isolated closed flat geodesics in this case. But we don't know the answer to Question 6.2.1 in \cite{BM} for the general case.

The paper is organized as follows. In section 2, we present some preliminaries and well known results. The proofs of Theorems \ref{A}, \ref{B} and \ref{C} will occupy Section 3. In Section 4, we prove Theorem \ref{Main} and ask a further related question.

\section{Preliminaries}

\subsection{Universal Cover}
Consider the universal covering space $\tilde{M}$ of $M$, which can be identified with the unit disk in the plane. The lifting of a geodesic $\gamma$ from $M$ to $\tilde{M}$ is denoted as $\tilde{\gamma}$. All the geodesics are supposed to have unit speed. It is well known that $\tilde{M}$ is a Hadamard manifold with many nice properties. For any two given points in $\tilde{M}$, there exists a unique geodesic joining them. Two geodesics $\tilde{\gamma_1}$ and $\tilde{\gamma_2}$ are said to be asymptotes if $d(\tilde{\gamma_1}(t), \tilde{\gamma_2}(t)) \leq C$ for some $C>0$ and $\forall t>0$. The asymptotes relation is an equivalence relation. Denote by $\tilde{M}(\infty)$ the set of all equivalence classes, which can be identified with the boundary of the unit disk. We denote $\tilde{\gamma}(+\infty)$ for the asymptote class of the geodesic $\tilde{\gamma}$, and $\tilde{\gamma}(-\infty)$ for the one of the reversed geodesic to $\gamma$.

Any closed geodesic $\gamma$ in $M$ can be lifted to a geodesic $\tilde{\gamma}$ on $\tilde{M}$, such that
$$\tilde{\gamma}(t+t_0)=\phi(\tilde{\gamma}(t)), \ \ \forall t \in \mathbb{R}$$
for some $t_0 >0$ and $\phi \in \pi_1(M)$. In this case, we say $\phi$ fixes $\tilde{\gamma}$ , i.e., $\phi(\tilde{\gamma})=\tilde{\gamma}$. Then $\phi$ acts on $\tilde{M}(\infty)$ in the natural way and fixes exactly two points $\tilde{\gamma}(\pm \infty)$. Moreover for any $x\in \tilde{M}(\infty)$ and $x \neq \tilde{\gamma}(\pm \infty)$, we have $\lim_{n\to +\infty}\phi^n(x)= \tilde{\gamma}(+\infty) $ and $\lim_{n\to -\infty}\phi^n(x)= \tilde{\gamma}(-\infty)$.

There are two continuous one dimensional distributions $E^s$ and $E^u$ on $T^1M$ which are invariant under the derivative of $\Phi^t$ (cf. \cite{E2}). Their integral manifolds form foliations $W^s$ and $W^u$ of $T^1M$ respectively which are invariant under $\Phi^t$, known as the stable and unstable horocycle foliations. The lifting of $W^s$ and $W^u$ to $T^1\tilde{M}$ are denoted as $\tilde{W}^s$ and $\tilde{W}^u$ respectively. If $w \in \tilde{W}^s(v)$, then geodesics $\tilde{\gamma}_v(t)$ and $\tilde{\gamma}_w(t)$ are asymptotic.

\subsection{Area of ideal triangles}
Given $x,y,z \in \tilde{M}(\infty)$, an ideal triangle with vertices $x, y, z$ means the region in $\tilde{M}$ bounded by the three geodesics joining $x$ and $y$, $y$ and $z$, $z$ and $x$. It is an interesting topic to study the area of ideal triangles. We have the following theorem due to Rafael Oswaldo Ruggiero(\cite{RR}):

\begin{theorem}\label{ideal triangle}
If $K(\tilde{\gamma}(t))\equiv 0$, for $\forall t\in \mathbb{R}$, then every ideal triangle having $\tilde{\gamma}(t)$ as an edge has infinite area.
\end{theorem}

In fact, if we have a triangle with vertices $x$, $a$, $b$, where $x=\tilde{\gamma}_1(+\infty)=\tilde{\gamma}_2(+\infty)$, $a \in \tilde{\gamma}_1$, $b\in \tilde{\gamma}_2$, and $\tilde{\gamma}_1$ is a flat geodesic, then the triangle has infinite area. The proof follows from the fact that the length of stable Jacobi fields decreases slowly along a geodesic with curvature close to zero.

\subsection{Flat strips}
A flat strip means a totally geodesic isometric imbedding $r:\mathbb{R}\times [0,c]\rightarrow \tilde{M}$, where $\mathbb{R}\times [0,c]$ is a strip in an Euclidean plane. We have the following Flat Strip Lemma due to Eberlein and O'Neill:
(\cite{EO})
\begin{lemma} \label{flat strip lemma}
If two distinct geodesics $\tilde{\alpha}$ and $\tilde{\beta}$ satisfy $d(\tilde{\alpha}(t), \tilde{\beta}(t)) < C$ for some $C>0$ and $\forall t\in \mathbb{R}$, then they are the boundary curves of a flat strip in $\tilde{M}$.
\end{lemma}

We also use the same name for the projection of a flat strip to $M$. An important progress toward the Conjecture \ref{conjecture} was made by Cao and Xavier(\cite{CX}) on the flat geodesics inside flat strips:

\begin{theorem} \label{flat strip closed}
A flat strip on $M$ consists of closed geodesics in the same homotopy type.
\end{theorem}

\section{Main Construction}

In this section, we mainly carry out two constructions based on a similar idea. First, we prove Theorem \ref{B} by constructing two points $y,z$ with the required property in the theorem starting from an aperiodic orbit of $x\in \Lambda$. Second, assume the contrary for theorem \ref{A}, i.e., there exist infinitely many periodic orbits, then we can construct an aperiodic orbit starting from them. Both constructions are based on the \emph{expansivity} property (cf. \cite{KH} Definition 3.2.11):

\begin{definition}
$x\in T^1M$ has the expansivity property if there exists a small $\delta_0>0$, such that if $d(\Phi^t(x), \Phi^t(y)< \delta_0$ for $\forall t\in \mathbb{R}$, then $y=\Phi^{t_0}(x)$ for some $t_0$ with $|t_0|<\delta_0$.
\end{definition}

\begin{lemma}\label{expansivity}
If $x$ is not tangent to a flat strip, it has the expansivity property.
\end{lemma}

\begin{proof}
Assume not. Then for an arbitrarily small $\epsilon >0$ less than the injectivity radius of $M$, there exists $y$ such that $y\notin \mathcal{O} (x)$ and $d(\gamma_x(t), \gamma_y(t))< \epsilon$ for $\forall t\in \mathbb{R}$. By the choice of $\epsilon$, we can lift $\gamma_x(t)$ and $\gamma_y(t)$ to the universal covering $\tilde{M}$ such that
$$d(\tilde{\gamma}_x(t), \tilde{\gamma}_y(t))< \epsilon \text{\ \ for\ } \forall t\in \mathbb{R}.$$

Thus by Lemma \ref{flat strip lemma}, $\tilde{\gamma}_x(t)$ and $\tilde{\gamma}_y(t)$ bound a flat strip. Hence $x$ is tangent to a flat strip, a contradiction.
\end{proof}

We first prove Theorem \ref{B} in the next subsection. Theorem \ref{C} is also proved there. After that we prove Theorem \ref{A}.

\subsection{Proof of Theorem \ref{B}}
Now we assume that $\Lambda \cap (\text{Per\ } \Phi)^c \neq \emptyset$, in other words, there exists an aperiodic orbit $\mathcal{O}(x)$ in $\Lambda$. We will construct $y, z$ as in Theorem \ref{B} starting from $\mathcal{O}(x)$. First we can always find two points on the orbit within a prescribed closeness:

\begin{lemma}\label{closelemma}
For any $k\in \mathbb{N}$, there exists two sequences of $t_k \to +\infty,$ and $t'_k\to +\infty$ such that $t'_k-t_k\to +\infty$ and
$$d(x_k, x'_k)< \frac{1}{k}, \ \ \ \ \text{where\ \  } x_k=\Phi^{t_k}(x), \ x'_k=\Phi^{t'_k}(x).$$
\end{lemma}

\begin{proof}
For any fixed $k\in \mathbb{N}$, let $\epsilon < \frac{1}{2k}$ sufficiently small be fixed. We choose a segment $[z_k, w_k]$ along the orbit $\mathcal{O}(x)$ from point $z_k$ to point $w_k$ with length $T_k$. Let $X$ be the vector field tangent to the geodesic flow on $T^1M$, and $X^{\bot}$ be the orthogonal complement of $X$, i.e. a two dimensional smooth distribution on $T^1M$. For any $y\in [z_k, w_k]$ define $D_\epsilon(y):=\exp_y(X^{\bot}_\epsilon(y))$, where $X^{\bot}_\epsilon(y)$ denote the $\epsilon$-ball centered at origin in the subspace $X^{\bot}(y)$.

 Assume $D_\epsilon(z) \cap D_\epsilon(w)= \emptyset$ for any $z, w \in [z_k, w_k]$. Since $T^1M$ is compact and its curvature is bounded, we have the following estimates on the volume:
$$C_0 \epsilon^2 T_k \leq \text{Vol}(\bigcup_{y\in [z_k, w_k]} D_\epsilon (y)) \leq \text{Vol} (T^1M).$$

But the above inequalities doesn't hold if we choose $T_k$ large enough. So there are two points in $[z_k, w_k]$, say, $x_k, x'_k$ such that $D_\epsilon(x_k) \cap D_\epsilon(x'_k)\neq \emptyset$, and hence $d(x_k, x'_k) < 2\epsilon < \frac{1}{k}$. Let $x_k=\Phi^{t_k}(x)$, $x'_k=\Phi^{t'_k}(x)$ where we can make $t'_k-t_k \to +\infty$ as $k\to +\infty$.

\end{proof}

For any pair of $x_k, x'_k$ with large enough $k$, we claim the expansivity in the positive direction of the flow:

\begin{proposition} \label {prop}
Fix an arbitrary small $\epsilon_0 >0$. There exists $s_k\to +\infty$, such that
$$d(\Phi^{s_k}(x_k), \Phi^{s_k}(x'_k))=\epsilon_0,$$
$$\text{\ and\ \ }d(\Phi^s(x_k), \Phi^s(x'_k))< \epsilon_0\ \ \ \text{\ for\ } \forall\ 0 \leq s< s_k.$$
\end{proposition}

\begin{remark}
In fact for the purpose of our construction, it is enough to have the expansivity in either positive or negative direction, and this is easily known since $x$ is not closed and hence not tangent to a flat strip by Theorem \ref{flat strip closed}. But in Proposition \ref{prop}, we have a stronger statement that the flow is expansive in the positive direction. To prove it, we will make use of several lemmas which seem to be of independent interest.
\end{remark}

The following lemma was proved in (\cite{BG}) and stated in (\cite{CX}):
\begin{lemma} \label{boundary}
If $w' \in W^s(w)$ and $\lim _{t\to +\infty} d(\gamma_w(t), \gamma_{w'}(t))=\delta >0$, then $\gamma_w(t)$ and $\gamma_{w'}(t)$ converge to the boundaries of a flat strip of width $\delta$.
\end{lemma}

\begin{proof}
Suppose $\lim_{s_i\to +\infty}\Phi^{s_i}(w) = v$ and $\lim_{s_i\to +\infty}\Phi^{s_i}(w')= v'$, then $v'\in W^s(v)$ and for any $t\in \mathbb{R}$:
$$d(\gamma_{v}(t), \gamma_{v'}(t))= \lim_{s_i\to +\infty}d(\gamma_{w}(t+s_i), \gamma_{w'}(t+s_i)) = \delta.$$

Hence we can lift the geodesics to $\tilde{M}$ such that $v'\in \tilde{W}^s(v)$ and $d(\tilde{\gamma}_v(t), \tilde{\gamma}_{v'}(t)) \\ = \delta$ for $\forall t \in \mathbb{R}$ (here we used the convexity of the function $d(\tilde{\gamma}_v(t), \tilde{\gamma}_{v'}(t))$). By Lemma \ref {flat strip lemma}, $\tilde{\gamma}_v(t)$ and $\tilde{\gamma}_{v'}(t)$ are the boundaries of a flat strip of width $\delta$.
\end{proof}

 The next lemma says that a flat geodesic converges to another closed geodesic(no matter flat or not), then the former must be closed as well and hence coincide with the latter.

\begin{lemma}\label{flat closed geodesic0}
If $y\in \Lambda$, and the $\omega$-limit set $\omega(y)= \mathcal{O}(z)$ where $\mathcal{O}(z)$ is periodic. Then $\mathcal{O}(y)=\mathcal{O}(z)$. In particular, $\mathcal{O}(y)$ is periodic.
\end{lemma}

\begin{proof}

\begin{figure}[b]
\centering
\setlength{\unitlength}{\textwidth}

\begin{tikzpicture}
\draw (0,0) circle (2);

\draw [name path=line 1](0,2) to [out=260,in=60] (-1,-1.73)node[below] {$\tilde{\gamma}_0$};
\draw  [name path=line 2](0,2)to [out=270,in=130] (1, -1.73)node[below] {$\tilde{\gamma}$};
\draw  [name path=line 3](0,2)to [out=280,in=135] (1.42, -1.42)node[right] {$\phi(\tilde{\gamma})$};

\draw [name path=line 4](-1.73, -1) to [out=355,in=182] (1.73,-1)node[right] {$\tilde{\alpha}$};
\draw ([name path=line 5]-2,0) to [out=355,in=182] (2,0) node[right] {$\phi(\tilde{\alpha})$};
\coordinate [label=below:$A$] (A) at (-0.67,-1.05);
\coordinate [label=below:$B$] (B) at (0.51,-1.05);
\coordinate [label=below left:$C$] (C) at (-0.35,-0.05);
\coordinate [label=below left:$D$] (D) at (0.15,-0.05);
\coordinate [label=below right:$E$] (E) at (0.5,-0.03);
\coordinate [label=above:$F$] (F) at (0,2);
\fill (A) circle (1.5pt)(B) circle (1.5pt)(C) circle (1.5pt)(D) circle (1.5pt)(E) circle (1.5pt)(F) circle (1.5pt);
\end{tikzpicture}

\caption[]{Proof of Lemma \ref{flat closed geodesic0}}

\end{figure}
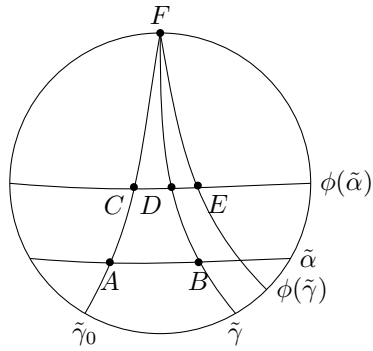

First we prove that we can lift geodesics $\gamma_z(t), \gamma_y(t)$ to the universal covering $\tilde{M}$, denoted as $\tilde{\gamma}_0(t)$ and $\tilde{\gamma}(t)$ respectively, such that $\tilde{\gamma}_0(+\infty)=\tilde{\gamma}(+\infty)$. Indeed, the assumption $\omega(y)= \mathcal{O}(z)$ guarantees that we can lift $\gamma_z(t), \gamma_y(t)$ to $\tilde{\gamma}_0(t)$ and $\tilde{\gamma}(t)$ such that $d(\tilde{\gamma}_0(kt_0), \tilde{\gamma}(t_k)) \to 0$ where $t_0$ is a period of $z$. Then by the convexity of $d(\tilde{\gamma_0}(t), \tilde{\gamma}(t))$ and a shifting of time on $\tilde{\gamma}(t)$ if necessary, we have $\lim_{t\to +\infty}d(\tilde{\gamma}_0(t), \tilde{\gamma}(t))=0$, hence $\tilde{\gamma}_0(+\infty)=\tilde{\gamma}(+\infty)$.

Since $\gamma_z(t)$ is a closed geodesic, there exist an isometry $\phi$ of $\tilde{M}$ such that $\phi(\tilde{\gamma}_0(t))=\tilde{\gamma}_0(t+t_0)$. Moreover, on the boundary of the disk $\tilde{M}(\infty)$, $\phi$ fixes exactly two points $\tilde{\gamma}_0(\pm \infty)$, and for any other point $a\in \tilde{M}(\infty)$, $\lim_{n\to +\infty}\phi^n(a)= \tilde{\gamma}_0(+\infty)$.

Assume $\tilde{\gamma}$ is not fixed by $\phi$. Then $\tilde{\gamma}$ and $\phi(\tilde{\gamma})$ don't intersect since $\phi(\tilde{\gamma})(+\infty)=\tilde{\gamma}(+\infty)$. We pick another geodesic $\tilde{\alpha}$ as shown in Figure 1. The image of infinite triangle $ABF$ under $\phi$ is the infinite triangle $CEF$. Since $\phi$ is an isometry, it preserves area. With a limit process, it is easy to show that Area of $ABCD$ $\geq$ Area of $DEF$. But since $\gamma$ is a flat geodesic, Area of $DEF$ is infinite by Theorem \ref{ideal triangle}, which is a contradiction since $ABCD$ has finite area. So $\phi(\tilde{\gamma})$ and $\tilde{\gamma}$ must coincide.

Hence $\tilde{\gamma}(\pm\infty)= \tilde{\gamma}_0(\pm\infty)$. Then either $\tilde{\gamma}(t)$ and $\tilde{\gamma}_0(t)$ bound a flat strip by Lemma \ref{flat strip lemma} or $\tilde{\gamma}(t)=\tilde{\gamma}_0(t)$. But $\lim_{t\to +\infty}d(\tilde{\gamma}(t), \tilde{\gamma}_0(t))=0$, hence $\tilde{\gamma}(t)=\tilde{\gamma}_0(t)$. Hence $\mathcal{O}(y)=\mathcal{O}(z)$.
\end{proof}

We improve Lemma \ref{flat closed geodesic0} as follows.
\begin{lemma}\label{flat closed geodesic}
 Suppose that $y\in \Lambda$ and $z\in \omega(y)$ where $z$ is periodic. Then $\mathcal{O}(y)=\mathcal{O}(z)$. In particular, $y$ is periodic.
\end{lemma}

\begin{proof}
Suppose that there exist $s_k \to +\infty$ such that $\Phi^{s_k}(y) \to z$. If $\Phi^{s_k}(y)\in W^s(z)$ for some $k$ then we must have $\omega(y)= \mathcal{O}(z)$. Then by Lemma \ref{flat closed geodesic0}, we have that $\mathcal{O}(\Phi^{s_k}(y))=\mathcal{O}(z)$. So we are done.

Suppose that $\Phi^{s_k}(y)\notin W^s(z)$ for any $k$. Note that if $y\neq z$ then $y$ and $z$ can not be tangent to a same flat strip. Therefore, for any large $k$ there exists small $\epsilon_0>0$ and a $l_k \to +\infty$ such that
$$d(\Phi^{l_k}(\Phi^{s_k}(y)), \Phi^{l_k}(z))=\epsilon_0,$$
where we take $l_k$ to be the smallest positive number to satisfy the above equality. By taking a subsequence but still using the same notation for simplicity, we assume that
\begin{equation}\label{e:sequence}
\Phi^{l_k}(\Phi^{s_k}(y)) \to y^+, \ \ \ \text{and\ \ \ \ }\Phi^{l_k}(z) \to z^+
\end{equation}
as $k\to +\infty$. Then $z^+$ is periodic and $d(y^+,z^+)=\epsilon_0$. For any $t>0$, since $0 < -t+l_k < l_k$ for large enough $k$, one has
$$d(\Phi^{-t}(y^+), \Phi^{-t}(z^+))=\lim_{k \to +\infty}d(\Phi^{-t+l_k+s_k}(y)),\Phi^{-t+l_k}(z) ) \leq \epsilon_0.$$
So $-y^+ \in W^s(-z^+)$. Replacing $y, z$ by $-y, -z$ respectively and applying the same argument, we can obtain two points $y^-,z^-$ such that $-y^- \in W^s(-z^-)$ and $d(y^-,z^-)=\epsilon_0$, $y^-\in \omega(-y)$ and $z^-$ is periodic. Then we have the following three different cases:
 \begin{enumerate}
 \item $\lim_{t\to \infty}d(\Phi^{t}(-y^+), \Phi^t(-z^+))= 0$. By Lemma \ref{flat closed geodesic0}, $-y^+$ is periodic and in fact $-y^+=-z^+$ as $\lim_{t\to \infty}d(\Phi^{t}(-y^+), \Phi^t(-z^+))= 0$. This contradicts to $d(y^+,z^+)=\epsilon_0$.
 \item $\lim_{t\to \infty}d(\Phi^{t}(-y^-), \Phi^t(-z^-))= 0$. By Lemma \ref{flat closed geodesic0}, $-y^-$ is periodic and in fact $-y^-=-z^-$ as $\lim_{t\to \infty}d(\Phi^{t}(-y^-), \Phi^t(-z^-))= 0$. This contradicts to $d(y^-,z^-)=\epsilon_0$.
 \item $\lim_{t\to \infty}d(\Phi^{t}(-y^+), \Phi^t(-z^+))= \delta_1$ and $\lim_{t\to \infty}d(\Phi^{t}(y^-), \Phi^t(z^-))= \delta_2$ for some $\delta_1, \delta_2>0$. By Lemma \ref{boundary} $-y^+$ converges to a closed flat geodesic. Then by Lemma \ref{flat closed geodesic0} $\gamma_{y^+}$ and $\gamma_z$ are boundaries of a flat strip of width $\delta_1$. By the same argument $\gamma_{y^-}$ and $\gamma_z$ are boundaries of a flat strip of width $\delta_2$. We claim that these two flat strips lie on the different sides of $\gamma_z$. Indeed, we choose $\epsilon_0$ small enough and consider the $\epsilon_0$ neighborhood of the closed geodesic $\gamma_z$ which contains two regions lying on the different sides of $\gamma_z$. We can choose the sequences in \eqref{e:sequence} for $y$ and $-y$ respectively such that $y^+$ and $y^-$ lie in different regions as above. This implies the claim. So we get a flat strip of width $\delta_1+\delta_2$ and $z$ is tangent to the interior of the flat strip. Now recall that $y^+ \in \omega(y)$ and $y^+$ is periodic, so we can apply all the arguments above to $y^+$ instead of $z$. Either we are arriving at a contradiction as in case (1) or case (2) and we are done, or we get a flat strip of width greater than $\delta_1+\delta_2$. But we can not enlarge a flat strip again and again in a compact surface. So we are done.
\end{enumerate}
\end{proof}
\begin{proof}[Proof of Theorem \ref{C}]
Assume that there exists a sequence $y_k \in \Lambda \cap (\text{Per\ } (\Phi))^c$ such that $\lim_{k \to +\infty}y_k=z$ for some $z\in \Lambda \cap \text{Per\ } (\Phi)$. We can apply the same argument in the proof of Lemma \ref{flat closed geodesic} replacing $\Phi^{s_k}(y)$ by $y_k$ to get a contradiction.
\end{proof}

\begin{proof}[Proof of Proposition \ref{prop}]
Assume the contrary, i.e. $d(\Phi^s(x_k), \Phi^s(x'_k))\leq \epsilon_0$ for $\forall s>0$. Then two geodesics $\gamma_{x_k}$ and $\gamma_{x'_k}$ are asymptotic. Without loss of generality, we suppose $x'_k \in W^s(x_k)$. By the convexity of $d(\gamma_{x_k}(t), \gamma_{x'_k}(t))$, we have either $\lim_{t\to +\infty}d(\gamma_{x_k}(t), \gamma_{x'_k}(t))=0$ or $\lim_{t\to +\infty}d(\gamma_{x_k}(t), \gamma_{x'_k}(t))=\delta>0$.
\begin{itemize}
           \item If $\lim_{t\to +\infty}d(\gamma_{x_k}(t), \gamma_{x'_k}(t))=0$, then we can choose a subsequence $s_i \to +\infty$, and $z$ such that $$\lim_{s_i\to +\infty} \Phi^{s_i}(x_k)=z$$
                and $$\lim_{s_i\to +\infty} \Phi^{s_i}(x'_k)=z.$$
                Since $x_k=\Phi^{t_k}(x)$ and $x'_k=\Phi^{t'_k}(x)$ with $t'_k-t_k\to +\infty$ as $k\to \infty$,
                we have $\lim_{s_i\to +\infty}\Phi^{s_i}(x'_k)=\lim_{s_i\to +\infty}\Phi^{t'_k-t_k}\circ\Phi^{s_i}(x_k)=\Phi^{t'_k-t_k}(z)$. Hence $\Phi^{t'_k-t_k}(z)=z$, so $z$ is a periodic point in $\Lambda$. As $z \in\omega(x_k)$, by Lemma \ref{flat closed geodesic}, $x_k$ is periodic, hence so is $x$. But we assume $x$ is aperiodic at the beginning. A contradiction.

     \item If $\lim_{t\to +\infty}d(\gamma_{x_k}(t), \gamma_{x'_k}(t))=\delta>0$, then $\omega(x_k)= \mathcal{O}(w)$ where $w$ is tangent to a boundary of a flat strip by Lemma \ref{boundary}. Then $w$ is periodic by Theorem \ref{flat strip closed}. Hence by Lemma \ref{flat closed geodesic0}, $x_k$ is periodic. A contradiction.

         \end{itemize}
So in each case we arrive at a contradiction, we are done.
\end{proof}

Now we continue with our construction.
\begin{proposition}\label{prop2}
For arbitrary small $\epsilon_0 >0$, there exist $a, b\in \Lambda \cap (\text{Per\ }(\Phi))^c$ such that
\begin{equation}\label{1}
 d(a, b)=\epsilon_0,
\end{equation}
\begin{equation}\label{2}
   d(\Phi^t(a), \Phi^t(b))\leq \epsilon_0\ \ \ \ \forall t<0,
\end{equation}
\begin{equation}\label{3}
a \notin \mathcal{O}(b),
\end{equation}
\begin{equation}\label{4}
a \in W^u(b).
\end{equation}

\end{proposition}

\begin{proof}

We apply Proposition \ref{prop}. We can pick a subsequence $k_i\to +\infty$, such that
$$\lim_{k_i\to +\infty}\Phi^{s_{k_i}}(x_{k_i})=a,$$
and $$\lim_{k_i\to +\infty}\Phi^{s_{k_i}}(x'_{k_i})=b.$$
Then $d(a, b)=\lim_{k_i\to +\infty}d(\Phi^{s_{k_i}}(x_{k_i}),\Phi^{s_{k_i}}(x'_{k_i}))=\epsilon_0.$ We get \eqref{1}.

For any $t<0$, since $0<s_{k_i}+t<s_{k_i}$ for large $k_i$, we have:
$$d(\Phi^t(a), \Phi^t(b))=\lim_{k_i\to +\infty}(d(\Phi^{s_{k_i}+t}(x_{k_i}),\Phi^{s_{k_i}+t}(x'_{k_i}))) \leq \epsilon_0.$$
Hence we get \eqref{2}.

Next suppose $a$ is periodic. Since
$$\lim_{k_i\to +\infty} \Phi^{t_{k_i}+s_{k_i}}(x)=\lim_{k_i\to +\infty}\Phi^{s_{k_i}}(x_{k_i})=a,$$ then $x$ is periodic by Lemma \ref{flat closed geodesic}. Contradiction. So $a\in (\text{Per\ }(\Phi))^c$. Similarly $b\in (\text{Per\ }(\Phi))^c$.  Hence $a, b\in \Lambda \cap (\text{Per\ }(\Phi))^c$.

Now we prove \eqref{3}, $a \notin \mathcal{O}(b)$. For a simpler notation, we assume
$$\lim_{k\to +\infty}\Phi^{s_{k}}(x_{k})=a,$$
and
$$\lim_{k\to +\infty}\Phi^{s_{k}}(x'_{k})=b.$$
We can lift $\gamma_{x_k}(t), \gamma_{x'_k}(t)$ on $M$ to geodesics $\tilde{\gamma}_k, \tilde{\gamma}'_k$ respectively on $\tilde{M}$ in the way such that $d(x_k, x'_k)<\frac{1}{k}$,
$d(y_k,y'_k)=\epsilon_0$, where $y_k=\Phi^{s_k}(x_k)$, $y'_k=\Phi^{s_k}(x'_k)$, and moreover $y_k \to a$, $y'_k \to b$. Then $\tilde{\gamma}_k$ converges to $\tilde{\gamma}=\tilde{\gamma}_a$, $\tilde{\gamma}'_k$ converges to $\tilde{\gamma}'=\tilde{\gamma}_b$ and $d(a, b)= \epsilon_0$. See Figure 2 (we use same notation for vector and its footpoint).

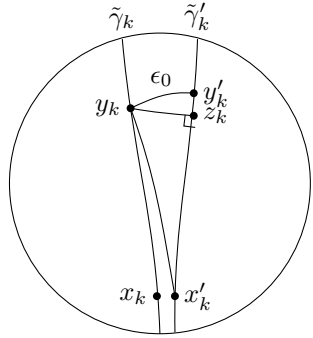
\begin{figure}[h]
\centering
\setlength{\unitlength}{\textwidth}

\begin{tikzpicture}
\draw (0,0) circle (2);

\draw (0,-2) to [out=92,in=275] (-0.5,1.92)node[above] {$\tilde{\gamma}_k$};
\draw (0.2,-1.98)to [out=91,in=268] (0.5, 1.92)node[above] {$\tilde{\gamma}'_k$};

\coordinate [label=left:$x_k$] (x_k) at (-0.04,-1.5);
\coordinate [label=right:$x'_k$] (x'_k) at (0.2,-1.5);
\coordinate [label=left:$y_k$] (y_k) at (-0.39,1);
\coordinate [label=right:$y'_k$] (y'_k) at (0.45,1.2);
\coordinate [label=right:$z_k$] (z_k) at (0.45,.9);

\draw (y_k) to [out=25,in=175] node[above ] {$\epsilon_0$}(y'_k);
\draw (x'_k) to [out=100,in=290] (y_k);
\draw (y_k) to [out=-10,in=175] (z_k);
\draw [right angle symbol={y_k}{z_k}{x'_k} size=.01];

\fill (x_k) circle (1.5pt) (x'_k) circle (1.5pt) (y_k)  circle (1.5pt) (y'_k)  circle (1.5pt)
(z_k)  circle (1.5pt);

\end{tikzpicture}
\caption[]{Proof of $\tilde{\gamma} \neq \tilde{\gamma}'$}

\end{figure}

First we show $d(y_k, \tilde{\gamma}_k')$ is bounded away from $0$. Denote $d_k:= d(y_k, \tilde{\gamma}'_k)=d(y_k, z_k)$, $l_k:=d(y_k, x'_k)$, $b_k:= d(x'_k, z_k)$, and $b'_k:=d(z_k, y'_k)$, and we already know that $d(x'_k, y'_k)=s_k$. Suppose $d_k \to 0$ as $k\to +\infty$. By triangle inequality, $\lim_{k\to +\infty}(l_k-b_k)=0$. But since $\lim_{k\to +\infty}(l_k-s_k) \leq \lim_{k \to+\infty}d(x_k, x'_k)=0$ and $s_k=b_k+b'_k$, we have $\lim_{k\to+\infty}b'_k=0$. But by triangle inequality, $\epsilon_0 < d_k+b_k' \to 0 $, a contradiction. Now$\tilde{\gamma} \neq \tilde{\gamma}'$ follows from $d(a, \tilde{\gamma}')=\lim_{k\to +\infty}d(y_k, \tilde{\gamma}'_k) \geq d_0$ for some $d_0 >0$.

Next we suppose there exists a $\phi \in \pi_1(M)$ such that $\phi(\tilde{\gamma})=\tilde{\gamma}'$. See Figure 3. Observe that $\tilde{\gamma}(-\infty)=\tilde{\gamma}'(-\infty)$ since $d(\Phi^t(a), \Phi^t(b))\leq \epsilon_0$, for $\forall t<0$. Let $\tilde{\gamma}_0$ be the closed geodesic such that $\phi(\tilde{\gamma}_0)=\tilde{\gamma}_0$. Then $\tilde{\gamma}(-\infty)=\tilde{\gamma}_0(-\infty)$. By Lemma \ref{flat closed geodesic0}, $\tilde{\gamma}$ is a closed geodesic, i.e. $a$ is a periodic point. We arrive at a contradiction. Hence for any $\phi \in \pi_1(M)$, $\phi(\tilde{\gamma})\neq\tilde{\gamma}'$. So $a \notin \mathcal{O}(b)$, and we get \eqref{3}.

At last, if $a \notin W^u(b)$, we can replace $a$ by some $a' \in \mathcal{O}(a)$, $b$ by some $b' \in \mathcal{O}(b)$ such that  $a'\in W^u(b')$ and the above three properties still hold for a different $\epsilon_0$. We get \eqref{4}.

\begin{figure}[h] \label{figure3}
\centering
\setlength{\unitlength}{\textwidth}

\begin{tikzpicture}
\draw (0,0) circle (2);

\draw (0,-2) to [out=102,in=275] (-0.5,1.92)node[above] {$\tilde{\gamma}$};
\draw (0,-2)to [out=91,in=250] (0.5, 1.92)node[above right] {$\tilde{\gamma}'=\phi(\tilde{\gamma})$};
\draw (0,-2) to [out=70,in=210] (2,0)node[right] {$\tilde{\gamma}_0$};

\coordinate [label=left:$a$] (a) at (-0.42,1);
\coordinate [label=right:$b$] (b) at (0.22,1);
\fill (a) circle (1.5pt) (b) circle (1.5pt);
\end{tikzpicture}
\caption[]{Proof of $\phi(\tilde{\gamma}) \neq \tilde{\gamma}'$}

\end{figure}
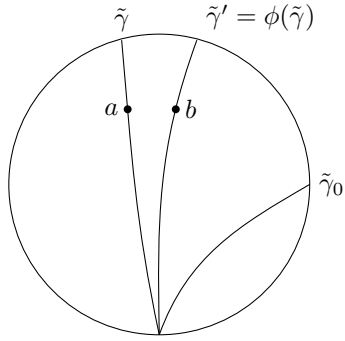

\end{proof}

\begin{proof}[Proof of Theorem \ref{B}]

We apply Proposition \ref{prop2}. Let $y=-a$, $z=-b$, then $y, z \in \Lambda \cap (\text{Per\ }(\Phi))^c$, $d(\Phi^t(y), \Phi^t(z))\leq \epsilon_0,\ \forall t>0$, $z \notin \mathcal{O}(y)$ and $y \in W^s(z)$.

If $\epsilon_0$ is small enough, we can lift geodesics $\gamma_y(t)$ and $\gamma_z(t)$ to $\tilde{\gamma}_y(t)$ and $\tilde{\gamma}_z(t)$ respectively on $\tilde{M}$ such that $d(\tilde{\gamma}_y(t),\tilde{\gamma}_z(t)) \leq \epsilon_0$ for any $t >0$ and $y\in \tilde{W}^s(z)$. Suppose $\lim_{t\to +\infty} d(\tilde{\gamma}_y(t), \tilde{\gamma}_z(t))=\delta >0$. Then by Lemma \ref{boundary}, $\tilde{\gamma}_y(t)$ and $\tilde{\gamma}_z(t)$ converge to the boundary of a flat strip, and hence $y$ and $z$ are periodic by Lemma \ref{flat closed geodesic0}, contradiction. So $\lim_{t\to +\infty} d(\tilde{\gamma}_y(t), \tilde{\gamma}_z(t))=0$. Hence $d(\Phi^t(y), \Phi^t(z))\to 0,  \ \ \text{as\ } t\to +\infty$.

\end{proof}

\subsection{Proof of Theorem \ref{A}}

Part of the proof of Theorem \ref{A} is a verbatim repetition of the one of Proposition \ref{prop2}, so we omit it.

\begin{proof}[Proof of Theorem \ref{A}]

Suppose $\Lambda \subset \text{Per\ }(\Phi)$. If $x\in \Lambda$, then $x$ is tangent to an isolated closed flat geodesic or a flat strip.

Assume the contrary. Then there exists a sequence of different vectors $x'_k \in \Lambda$ such that $\lim_{k\to +\infty}x'_k=x$ for some $x\in \Lambda$. Here different $x'_k$ are tangent to different isolated closed geodesics or to different flat strips, and $x$ is tangent to an isolated closed geodesic or to a flat strip. For large enough $k$, we suppose $d(x'_k, x) <\frac{1}{k}$. Fix any small $\epsilon_0 >0$. It is impossible that $d(\Phi^t(x'_k), \Phi^t(x)) \leq \epsilon_0$ for $\forall t>0$. Otherwise, $\tilde{\gamma}_{x'_k}(t), \tilde{\gamma}_x(t)$ are positively asymptotic closed geodesics so they must coincide by Lemma \ref{flat closed geodesic0}. Hence there exists a $s_k\to +\infty$, such that
$$d(\Phi^{s_k}(x'_k), \Phi^{s_k}(x))=\epsilon_0,$$
and
$$d(\Phi^{s}(x'_k), \Phi^{s}(x))\leq \epsilon_0\ \ \forall 0\leq s <s_k.$$

Denote $y_k:=\Phi^{s_k}(x)$ and $y'_k:=\Phi^{s_k}(x'_k)$. Without loss of generality, suppose $y_k\to a$ and $y'_k\to b$. Similar proof as in Proposition \ref{prop2} gives $d(a, b)=\epsilon$ and $d(\Phi^t(a), \Phi^t(b)) \leq \epsilon_0$ for $\forall t\leq 0$. If we lift the geodesics to $\tilde{M}$ (using the same notation as in the proof of Proposition \ref{prop2}), we can prove $\tilde{\gamma} \neq \tilde{\gamma}'$ similarly. But then we have two flat closed geodesics $\tilde{\gamma}$ and $\tilde{\gamma}'$ that are negatively asymptotic, so they must coincide by Lemma \ref{flat closed geodesic0}. A contradiction.

\end{proof}

\section{Proof of Theorem \ref{Main}}

 We shall prove Theorem \ref{Main} by arguing that the second of the dichotomy cannot happen if $\{p\in M: K(p)=0\}^c$ has only finitely many components.

\begin{proof}[Proof of Theorem \ref{Main}]

Suppose $\Lambda \cap (\text{Per\ }(\Phi))^c \neq \emptyset$. Consider the two points $y$ and $z$ given by Theorem \ref{B}. We lift the geodesics $\gamma_y(t)$ and $\gamma_z(t)$ to the universal covering $\tilde{M}$, denoted as $\tilde{\gamma}_1$ and $\tilde{\gamma}_2$ respectively.

Consider the connected components of $\{p\in M: K(p)<0\}$ lifted to $\tilde{M}$ and we want to see how they distribute inside the ideal triangle bounded by $\tilde{\gamma}_1$ and $\tilde{\gamma}_2$. Since $\tilde{\gamma}_1$ and $\tilde{\gamma}_2$ are flat geodesics, any connected component doesn't intersect $\tilde{\gamma}_1$ or $\tilde{\gamma}_2$. Since the number of the connected components on $M$ is finite, the radii of their inscribed circles are bounded away from $0$. When lifted to the universal covering, the sizes of the connected components do not change. But $d(\tilde{\gamma}_1(t), \tilde{\gamma}_2(t))\to 0$ as $t \to +\infty$, we can claim that the connected components on $\tilde{M}$ cannot approach $w$ inside of the ideal triangle. See Figure \ref{1.3}.

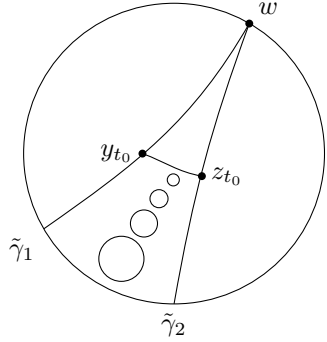
\begin{figure}[h] \label{1.3}

\centering
\setlength{\unitlength}{\textwidth}

\begin{tikzpicture}
\draw (0,0) circle (2);

\coordinate [label=left:$y_{t_0}$] (y_{t_0}) at (-.42,0);
\coordinate [label=right:$z_{t_0}$] (z_{t_0}) at (.37,-.3);
\coordinate [label=above right:$w$] (w) at (1,1.73);
\fill (z_{t_0}) circle (1.5pt) (y_{t_0}) circle (1.5pt) (w) circle (1.5pt);

\draw (-1.73,-1) node[below left] {$\tilde{\gamma}_1$} to [out=35,in=240] (1,1.73);
\draw (0,-2) node[below] {$\tilde{\gamma}_2$} to [out=80,in=250] (1,1.73);
\draw (y_{t_0}) to [out=340,in=170] (z_{t_0});

\draw (-.7,-1.4) circle (0.3);
\draw (-.4,-.93) circle (.18);
\draw (-.2,-.6) circle (.12);
\draw (-.01,-.35) circle (.08);

\end{tikzpicture}
\caption[]{Proof of Theorem \ref{Main}}

\end{figure}

So there exist a $t_0 >0$, $y_{t_0}=\Phi^{t_0}(y)$, $z_{t_0}=\Phi^{t_0}(z)$,  such that the infinite triangle $z_{t_0}y_{t_0}w$ is a flat region. Then $d(\Phi^t(y), \Phi^t(z))\equiv d((y_{t_0}, z_{t_0})$ for all $t \geq t_0$ . Indeed, if we construct a geodesic variation between $\tilde{\gamma}_1$ and $\tilde{\gamma}_2$, then Jacobi fields are constant for $t\geq t_0$ since $K\equiv 0$, thus $d(\tilde{\gamma}_1(t), \tilde{\gamma}_2(t))$ is constant when $t \geq t_0$. We get a contradiction since $d(\Phi^t(y), \Phi^t(z))\to 0$ as $t \to +\infty$ by Theorem \ref{B}.

Finally we conclude that $\Lambda \subset \text{Per\ }(\Phi)$. In particular the geodesic flow is ergodic by Theorem \ref{A}.
\end{proof}

At last, let us suppose that $\{p\in M: K(p) <0\}$ has infinitely many connected components. By the argument in the proof of Theorem \ref{Main}, we know that the infinite triangle in Figure \ref{1.3} contains at most finitely many liftings of a same single connected component. But we don't know if there are still only finitely many liftings of all different connected components since the size of connected components could be arbitrarily small.

\begin{question} \label{inf}
If $\{p\in M: K(p) <0\}$ has infinitely many connected components, is it possible that $\lim_{t\to +\infty}d(\Phi^t(y), \Phi^t(z))=0$ for some $y, z \in \Lambda$, $y\notin \mathcal{O} (z)$?
\end{question}

A negative answer to Question \ref{inf} together with Theorem \ref{A} will imply Conjecture \ref{conjecture}, and in particular the ergodicity of the geodesic flow.
\\[5mm]
\textbf{Acknowledgement.} The author would like to thank Federico Rodriguez Hertz for posing the problem and numerous discussions with him. The author would also like to express gratitude to Anatole Katok for his comments and constant encouragement. He also thanks Barbara Schapira, Keith Burns and Gerhard Knieper for valuable comments.

\end{document}